\documentclass[leqno,12pt]{amsart}
\usepackage{amsfonts}
\usepackage{amsmath,amssymb,amsthm}

\everymath{\displaystyle}

\newcommand{\E}{\mathbb E}

\newcommand{\tr}{\mathrm{tr}}
\newcommand{\ds}{\displaystyle}

\newtheorem{theorem}{Theorem}[section]

\theoremstyle{definition}

\theoremstyle{remark}

\numberwithin{equation}{section}

\begin{document}

\title[Timelike General Rotational Surfaces]{Basic Classes of Timelike General Rotational Surfaces in the Four-dimensional Minkowski Space}

\author{Victoria Bencheva, Velichka Milousheva}

\address{Institute of Mathematics and Informatics, Bulgarian Academy of Sciences,
Acad. G. Bonchev Str. bl. 8, 1113, Sofia, Bulgaria}
\email{viktoriq.bencheva@gmail.com}
\email{vmil@math.bas.bg}


\begin{abstract}

In the present paper, we consider timelike general rotational surfaces in the Minkowski 4-space which are analogous to the general rotational surfaces in the Euclidean 4-space introduced by C. Moore. We study two types of such surfaces (with timelike and spacelike meridian curve, respectively) and describe analytically some of their basic geometric classes: flat timelike general rotational surfaces, timelike general rotational surfaces with flat normal connection, and timelike general rotational surfaces with non-zero constant mean curvature. We give explicitly all minimal timelike general rotational surfaces and all timelike general rotational surfaces with parallel normalized mean curvature vector field.

\end{abstract}

\maketitle

\section{Introduction}

Rotational surfaces are basic sources of examples of many geometric classes of surfaces both in Euclidean and pseudo-Euclidean spaces. In \cite{M}, C. Moore introduced a class of  surfaces in the four-dimensional Euclidean space $\mathbb R^4$ which generalized the rotational surfaces and described a special case of such surfaces with constant Gauss curvature \cite{M2}.

The analogue of these surfaces in the Minkowski 4-space was considered by G. Ganchev and the second author in \cite{GM6}, where spacelike  general rotational surfaces in $\mathbb R^4_1$ with plane meridian curves and special invariants were studied. The flat general rotational surfaces and the general rotational surfaces with flat normal connection were described analytically and  the minimal general rotational surfaces and the general rotational surfaces consisting of parabolic points were completely classified. Spacelike general rotational surfaces in  $\mathbb R^4_1$  with plane meridian curves  and having pointwise 1-type Gauss map  were studied by U. Dursun in \cite{Dur}. 

Analogously to the general rotational surfaces in the Euclidean 4-space and  in the Minkowski 4-space, in \cite{ATM}, Y. Aleksieva, N.-C. Turgay and the second author defined general rotational surfaces of elliptic and hyperbolic type in the pseudo-Euclidean 4-space with neutral metric  $\mathbb R^4_2$. Especially, Lorentz general rotational surfaces with plane  meridian curves were considered and the complete classification of some special geometric classes was given: minimal general rotational surfaces of elliptic and hyperbolic type,   general rotational surfaces with parallel normalized mean curvature vector field, flat  general rotational surfaces, and general rotational surfaces with flat normal connection.

In the present paper, we consider two types of timelike general rotational surfaces in the Minkowski 4-space $\mathbb R^4_1$ with meridian curves lying in 2-dimensional planes and describe analytically some of their basic geometric classes. In Theorem \ref{T:flat}  we describe all flat timelike general rotational surfaces of first and second type. In Theorem \ref{T:flat_normal_connection}  we describe the timelike general rotational surfaces  with flat normal connection. The  minimal general rotational surfaces of first and second type are explicitly determined in Theorem \ref{T:minimal}. In Theorem \ref{T:CMC} we classify the timelike general rotational surfaces with non-zero constant mean curvature. In the last section, we give explicitly all timelike general rotational surfaces with parallel normalized mean curvature vector field.

\section{Preliminaries}

Let $\mathbb R^4_1$  be the four-dimensional Minkowski space  endowed with the metric
$\langle , \rangle$ of signature $(3,1)$ and  $Oe_1e_2e_3e_4$ be a
fixed orthonormal coordinate system, i.e. $e_1^2 =
e_2^2 = e_3^2 = 1, \, e_4^2 = -1$, giving the orientation of
$\mathbb R^4_1$. The standard flat metric is given in local coordinates by
$dx_1^2 + dx_2^2 + dx_3^2 -dx_4^2.$

A surface $M^2: z = z(u,v), \, \, (u,v) \in {\mathcal D}$
(${\mathcal D} \subset \mathbb R^2$) in $\mathbb R^4_1$ is said to be
\emph{spacelike} if $\langle , \rangle$ induces  a Riemannian
metric $g$ on $M^2$. A surface $M^2$ is said to be
\emph{timelike} if the induced metric $g$ on $M^2$ is a metric with index 1. So,  at each point $p$ of a spacelike (resp. timelike) surface
$M^2$ we have the following decomposition:
$$\mathbb R^4_1 = T_pM^2 \oplus N_pM^2$$
with the property that the restriction of the metric $\langle ,
\rangle$ onto the tangent space $T_pM^2$ is of signature $(2,0)$ (resp.  $(1,1)$),
and the restriction of the metric $\langle , \rangle$ onto the
normal space $N_pM^2$ is of signature $(1,1)$ (resp. $(2,0)$).

Denote by $\widetilde{\nabla}$ and $\nabla$ the Levi Civita connections on $\mathbb R^4_1$ and $M^2$, respectively.
If $x$ and $y$ are vector fields tangent to $M^2$ and $\xi$ is a normal vector field, then we have the following formulas of Gauss and Weingarten:
$$\begin{array}{l}
\vspace{2mm}
\widetilde{\nabla}_xy = \nabla_xy + \sigma(x,y);\\
\vspace{2mm}
\widetilde{\nabla}_x \xi = - A_{\xi} x + D_x \xi,
\end{array}$$
which define the second fundamental tensor $\sigma$, the normal connection $D$
and the shape operator $A_{\xi}$ with respect to $\xi$.

The mean curvature vector  field $H$ of $M^2$ is defined as
$H = \ds{\frac{1}{2}\,  \tr\, \sigma}$. 
A normal vector field $\xi$ on a surface $M^2$ is called \emph{parallel in the normal bundle} (or simply \emph{parallel}) if $D{\xi}=0$  \cite{Chen}.
The surface $M^2$ is said to have \emph{parallel mean curvature vector field} if its mean curvature vector $H$ is parallel, i.e.
$D H =0$. The class of surfaces with parallel mean curvature vector field is naturally extended to the class of surfaces with parallel
normalized mean curvature vector field as follows: a surface  is said to have \textit{parallel normalized mean curvature vector field} if 
 $H$ is non-zero and  there exists a unit vector field in the direction of $H$ 
which is parallel in the normal bundle \cite{Chen-MM}.

\vskip 2mm
Let $M^2: z=z(u,v), \,\, (u,v) \in \mathcal{D}$ $(\mathcal{D} \subset \mathbb R^2)$  be a local parametrization on a
timelike surface in $\mathbb R^4_1$.
The tangent space $T_pM^2$ at an arbitrary point $p=z(u,v)$ of $M^2$ is  spanned by $z_u$ and $z_v$. We use the standard denotations
$E(u,v)=\langle z_u,z_u \rangle, \; F(u,v)=\langle z_u,z_v
\rangle, \; G(u,v)=\langle z_v,z_v \rangle$ for the coefficients
of the first fundamental form.
Since $M^2$ is timelike, without loss of generality we assume that
$\langle z_u,z_u \rangle < 0$, $\langle z_v,z_v \rangle > 0$.
Hence,
$E(u,v)<0, \; G(u,v)>0$ and we set $W=\sqrt{- EG+F^2}$.
We choose an orthonormal frame field $\{n_1, n_2\}$ of the normal bundle, i.e. $\langle
n_1, n_1 \rangle =1$, $\langle n_2, n_2 \rangle = 1$, $\langle n_1, n_2 \rangle = 0$.
Then we have the following derivative formulas:
\begin{equation}\label{E:Eq-1}
\begin{array}{l}
\vspace{2mm} \widetilde{\nabla}_{z_u}z_u=z_{uu} = - \Gamma_{11}^1 \, z_u +
\Gamma_{11}^2 \, z_v + c_{11}^1\, n_1 + c_{11}^2\, n_2;\\
\vspace{2mm} \widetilde{\nabla}_{z_u}z_v=z_{uv} = - \Gamma_{12}^1 \, z_u +
\Gamma_{12}^2 \, z_v + c_{12}^1\, n_1 + c_{12}^2\, n_2;\\
\vspace{2mm} \widetilde{\nabla}_{z_v}z_v=z_{vv} = - \Gamma_{22}^1 \, z_u +
\Gamma_{22}^2 \, z_v + c_{22}^1\, n_1 + c_{22}^2\, n_2;\\
\end{array} 
\end{equation}
where $\Gamma_{ij}^k$ are the Christoffel's symbols and the functions $c_{ij}^k, \,\, i,j,k = 1,2$  are given by
$$\begin{array}{lll}
\vspace{2mm}
c_{11}^1 = \langle z_{uu}, n_1 \rangle; & \qquad  c_{12}^1 = \langle z_{uv}, n_1 \rangle; &  \qquad  c_{22}^1 = \langle z_{vv}, n_1 \rangle; \\
\vspace{2mm}
c_{11}^2 = \langle z_{uu}, n_2 \rangle; & \qquad  c_{12}^2 = \langle z_{uv}, n_2 \rangle;& \qquad c_{22}^2 = \langle z_{vv}, n_2 \rangle.
\end{array}$$

It is obvious, that  $M^2$ lies in a two-dimensional plane if and only if
it is totally geodesic, i.e. $c_{ij}^k=0$ for all $i,j,k = 1, 2.$ 
Further, we assume that at least one of the coefficients $c_{ij}^k$ is not
zero.

Let us consider the following determinants:
$$
\Delta_1 = \left\vert%
\begin{array}{cc}
\vspace{2mm}
  c_{11}^1 & c_{12}^1 \\
  c_{11}^2 & c_{12}^2 \\
\end{array}%
\right\vert, \quad
\Delta_2 = \left\vert%
\begin{array}{cc}
\vspace{2mm}
  c_{11}^1 & c_{22}^1 \\
  c_{11}^2 & c_{22}^2 \\
\end{array}%
\right\vert, \quad
\Delta_3 = \left\vert%
\begin{array}{cc}
\vspace{2mm}
  c_{12}^1 & c_{22}^1 \\
  c_{12}^2 & c_{22}^2 \\
\end{array}%
\right\vert.
$$
At a given point $p \in M^2$, the \textit{first normal space} of $M^2$  in $\E^4_1$, denoted by
$\rm{Im} \, \sigma_p$, is the subspace given by
$${\rm Im} \, \sigma_p = {\rm span} \{\sigma(x, y): x, y \in T_p M^2 \}.$$

It is obvious, that the condition $\Delta_1 = \Delta_2 = \Delta_3 = 0$  characterizes points at which
the first normal space  ${\rm Im} \, \sigma_p$ is one-dimensional. Such points are called \textit{flat} (or \textit{inflection}) points of the surface \cite{Lane, Little}.
E. Lane \cite{Lane} has shown that  every point of a surface in  a 4-dimensional affine space $\mathbb{A}^4$ is an inflection point
if and only if the surface is developable or lies in a 3-dimensional space.
So, further we consider timelike surfaces free of inflection  points, i.e. we assume that $(\Delta_1, \Delta_2, \Delta_3) \neq (0,0,0)$.

\vskip 2mm
Now, we shall consider  a special class of surfaces in the Minkowski 4-space, which are called general rotational surfaces.

General rotational surfaces  in the Euclidean 4-space $\mathbb R^4$ were
introduced  by F. N. Cole \cite{Co} and later studied by  C. Moore \cite{M}. We present shortly the construction.  Let $c: x(u) = \left(
x^1(u), x^2(u),  x^3(u), x^4(u)\right)$; $ u \in J \subset \mathbb R$ be
a smooth curve in $\mathbb R^4$, and $\alpha$, $\beta$ be real constants. A
general rotation of the meridian curve $c$ in $\mathbb R^4$ is defined by
$$X(u,v)= \left( X^1(u,v), X^2(u,v),  X^3(u,v), X^4(u,v)\right),$$
where
$$\begin{array}{ll}
\vspace{2mm} 
X^1(u,v) = x^1(u)\cos\alpha v - x^2(u)\sin\alpha v; & \qquad X^3(u,v) = x^3(u)\cos\beta v - x^4(u)\sin\beta v; \\
\vspace{2mm} 
 X^2(u,v) = x^1(u)\sin\alpha v + x^2(u)\cos\alpha v; & \qquad X^4(u,v) = x^3(u)\sin\beta v + x^4(u)\cos\beta v,\\
\end{array}$$
$v \in [0;2\pi)$. The constants $\alpha$ and $\beta$ determine the rates of rotation.
In the case $\beta = 0$, $x^2(u) = 0$, the plane $Oe_3e_4$ is fixed and one gets
the classical rotation about a fixed two-dimensional axis. In \cite{M2}, C. Moore described a special case of general rotational surfaces with constant Gauss curvature.

In \cite{GM2-a}, the second author considered  a special case of such surfaces,
given by
\begin{equation} \label{E:Eq-5}
\mathcal{M}: z(u,v) = \left(
f(u) \cos\alpha v, f(u) \sin \alpha v, g(u) \cos \beta v, g(u)
\sin \beta v \right),
\end{equation}
where $u \in J \subset \mathbb R, \,\,  v \in [0;
2\pi)$, $f(u)$ and $g(u)$ are smooth functions, satisfying
$\alpha^2 f^2(u)+ \beta^2 g^2(u)
> 0 , \,\, f'\,^2(u)+ g'\,^2(u) > 0$, and $\alpha, \beta$ are positive
constants. In the case  $\alpha \neq \beta$ each parametric curve
$u = const$ is a curve in $\mathbb R^4$ with constant Frenet curvatures,
and in the case  $\alpha = \beta$ each parametric curve $u =
const$ is a circle. The parametric curves  $v = const$  are plane curves called the meridians of $\mathcal{M}$.

The surfaces defined by \eqref{E:Eq-5} are general rotational surfaces in
the sense of C. Moore with plane meridian curves. In \cite{GM2-a}, the second author found the
invariants of these surfaces and completely  classified the minimal
super-conformal general rotational surfaces in $\mathbb R^4$. 

\vskip 2mm
Similarly to the general rotations in $\mathbb R^4$ one can consider
general rotational surfaces  in the Minkowski 4-space $\mathbb R^4_1$.
 Let $c: x(u) = \left( x^1(u), x^2(u),  x^3(u), x^4(u)\right)$; $ u \in J \subset \mathbb R$ be
a smooth spacelike or timelike curve in $\mathbb R^4_1$, and $\alpha$, $\beta$ be real constants.
We consider the surface defined by
\begin{equation} \label{E:Eq-2}
X(u,v)= \left( X^1(u,v), X^2(u,v),  X^3(u,v), X^4(u,v)\right),
\end{equation}
where
\begin{equation} \label{E:Eq-3}
\begin{array}{ll}
\vspace{2mm} 
X^1(u,v) = x^1(u)\cos\alpha v - x^2(u)\sin\alpha v; & \qquad X^3(u,v) = x^3(u)\cosh\beta v + x^4(u)\sinh\beta v; \\
\vspace{2mm}
X^2(u,v) = x^1(u)\sin\alpha v + x^2(u)\cos\alpha v; & \qquad X^4(u,v) = x^3(u)\sinh\beta v + x^4(u)\cosh\beta v. \\
\end{array}
\end{equation}

In the case $\beta = 0$, $x^2(u) = 0$ (or  $x^1(u) = 0$)   one gets
the standard rotational surface (with two-dimensional axis) of elliptic type in $\mathbb R^4_1$.
A local classification of spacelike rotational surfaces of elliptic type, whose
mean curvature vector field is either vanishing or lightlike, was obtained in \cite{Haesen-Ort-2}.

In the case $\alpha = 0$, $x^3(u) = 0$  one gets
the standard hyperbolic rotational surface of first type, and in the case
$\alpha = 0$, $x^4(u) = 0$ we get the standard hyperbolic rotational surface of second type.
Spacelike rotational surfaces of hyperbolic type
with either vanishing or lightlike mean curvature vector field are classified in \cite{Haesen-Ort-1}.
In \cite{Liu-Liu-1}, the classification of timelike and spacelike hyperbolic rotational surfaces
with non-zero constant mean curvature in the three-dimensional de Sitter space
$\mathbb{S}^3_1$ is given.
Spacelike and timelike Weingarten rotational surfaces in $\mathbb{S}^3_1$
are studied in \cite{Liu-Liu-2}.
The class of Chen spacelike rotational surfaces of hyperbolic or elliptic type in $\mathbb R^4_1$ is described in \cite{GM3}.  
Timelike rotational surfaces with two-dimensional axis of hyperbolic,  elliptic, and parabolic type in $\mathbb R^4_1$ were studied in \cite{Bek-Dur}, where  such surfaces with pointwise 1-type Gauss map of first and second kind  were considered.

In the case $\alpha > 0$ and $\beta > 0$ the surfaces defined by \eqref{E:Eq-2} and \eqref{E:Eq-3} are analogous to the general rotational surfaces of C. Moore in $\mathbb R^4$.
In \cite{GM5}  and \cite{GM6}, G. Ganchev and the second author considered  spacelike general rotational surfaces with plane meridian curves in $\mathbb R^4_1$ and  described analytically some basic geometric classes of these surfaces.

\vskip 2mm
In the present paper,  we  study timelike general rotational surfaces in the Minkowski space $\mathbb R^4_1$ with plane meridian curves and describe analytically some of their basic geometric classes, namely:  flat surfaces, surfaces with flat normal connection, minimal surfaces, surfaces with non-zero constant mean curvature, and surfaces
with parallel normalized mean curvature vector field.

\section{Timelike general rotational surfaces with plane meridian curves} 

 We consider a surface $\mathcal{M}_1$  in $\mathbb R^4_1$ parametrized by
\begin{equation} \label{E:Eq-6}
\mathcal{M}_1: z(u,v) = \left( f(u) \cos \alpha v, f(u) \sin \alpha v, g(u) \sinh \beta v, g(u) \cosh \beta v \right),
\end{equation}
where $u \in J \subset \mathbb R$, $v \in [0; 2\pi)$, $f(u)$ and $g(u)$ are smooth functions, satisfying the conditions 
$$f'\,^2(u)- g'\,^2(u) < 0, \quad \alpha^2 f^2(u)+ \beta^2 g^2(u) > 0,$$
and $\alpha, \beta$ are positive constants. This is a general rotational surface with plane meridian curves, defined by  \eqref{E:Eq-2}  and   \eqref{E:Eq-3} in the case $x_2(u) =0$ and $x_3(u) =0$. The meridian curve  $c: x(u) = \left(f(u), 0,  0, g(u)\right)$; $ u \in J \subset \mathbb R$  is timelike. 

The coefficients of the first fundamental form of $\mathcal{M}_1$ are
$$E = f'\,^2(u)- g'\,^2(u); \quad F = 0; \quad G =\alpha^2 f^2(u)+ \beta^2 g^2(u).$$
$\mathcal{M}_1$ is a timelike surface in $\mathbb R^4_1$, since $E<0, \; G>0$. We call it  a \textit{general rotational surface of first type}.

We consider the tangent frame field $\{x, y\}$ defined by:
\begin{equation} \label{E:Eq-xy}
x = \frac{z_u}{\sqrt{g'\,^2 -f'\,^2}}; \qquad y = \frac{z_v}{\sqrt{\alpha^2 f^2 + \beta^2 g^2}}.
\end{equation}
Obviously, $\langle x, x \rangle = -1; \, \langle y, y \rangle = 1; \, \langle x, y \rangle = 0$.
Let us consider the following normal frame field of $\mathcal{M}_1$:
\begin{equation} \label{E:Eq-normal}
\begin{array}{l}
\vspace{2mm} n_1 = \ds{\frac{1}{\sqrt{\alpha^2 f^2 + \beta^2 g^2}}\left( \beta g \sin \alpha v,- \beta g \cos \alpha v, \alpha
f \cosh \beta v,  \alpha f \sinh \beta v\right)};\\
\vspace{2mm} n_2 = \ds{\frac{1}{\sqrt{g'\,^2 -f'\,^2}}\left(g'
\cos \alpha v, g' \sin \alpha v,  f' \sinh \beta v, f' \cosh \beta
v  \right)}.
\end{array}
\end{equation}
It can easily be seen that $\langle n_1, n_1 \rangle = \langle n_2, n_2 \rangle = 1; \langle n_1, n_2 \rangle =0$.
Hence, $\{x,y, n_1, n_2\}$ is an orthonormal moving frame field of the surface  $\mathcal{M}_1$.

By direct computation we obtain the second partial derivatives of $z(u,v)$:
\begin{equation} \label{E:Eq-second}
\begin{array}{l}
\vspace{2mm} 
z_{uu}= \left( f''(u) \cos \alpha v, f''(u) \sin \alpha v, g''(u) \sinh \beta v, g''(u) \cosh \beta v \right);\\
\vspace{2mm} 
z_{uv}= \left( - \alpha f'(u)  \sin \alpha v, \alpha  f'(u) \cos \alpha v,  \beta g'(u) \cosh \beta v,  \beta g'(u) \sinh \beta v \right);\\
\vspace{2mm} 
z_{vv}= \left( - \alpha^2 f(u) \cos \alpha v,  - \alpha^2 f(u) \sin \alpha v, \beta^2 g(u) \sinh \beta v, \beta^2 g(u) \cosh \beta v \right).
\end{array}
\end{equation}
Formulas \eqref{E:Eq-normal} and \eqref{E:Eq-second} imply that the functions $c^k_{ij}$, $i,j,k =1,2$ are expressed as follows:
\begin{equation} \label{E:Eq-4}
\begin{array}{lll}
\vspace{2mm}
c_{11}^1 = 0; & \qquad c_{12}^1 = \ds{\frac{\alpha \beta (f g'- f'g)}{\sqrt{\alpha^2 f^2 + \beta^2 g^2}}};  & \qquad c_{22}^1 = 0; \\
\vspace{2mm}
c_{11}^2 = \ds{\frac{f'' g'- f'g''}{\sqrt{g'\,^2 -f'\,^2}}};  & \qquad c_{12}^2 = 0;  & \qquad c_{22}^2 = - \ds{\frac{\alpha^2 f g'+ \beta^2 f'g}{\sqrt{g'\,^2 -f'\,^2}}}.
\end{array} 
\end{equation}
Taking into consideration \eqref{E:Eq-xy} and \eqref{E:Eq-4}, we obtain that the second fundamental tensor $\sigma$ is expressed by:
$$
\begin{array}{l}
\sigma(x, x) = \ds{\frac{f'' g'- f'g''}{(g'\,^2 -f'\,^2)^{\frac{3}{2}}} \, n_2};\\
\sigma(x, y) = \ds{\frac{\alpha \beta (f g'- f'g)}{\sqrt{g'\,^2 -f'\,^2}(\alpha^2 f^2 + \beta^2 g^2)}\, n_1};\\
\sigma(y, y) = - \ds{\frac{\alpha^2 f g'+ \beta^2 f'g}{\sqrt{g'\,^2 -f'\,^2}(\alpha^2 f^2 + \beta^2 g^2)} \, n_2}.
\end{array} 
$$
Using the last formulas, we obtain the Gauss curvature $K$ and the mean curvature vector field $H$ of $\mathcal{M}_1$ expressed by the functions $f(u), g(u)$  and their derivatives:
\begin{equation} \label{E:Eq-K}
K = \frac{\left( f'' g'- f'g'' \right)\left( \alpha^2 f g'+ \beta^2 f'g \right)(\alpha^2 f^2 + \beta^2 g^2)+ \alpha^2 \beta^2 \left( f g'- f'g \right)^2 (g'\,^2 - f'\,^2)}{(g'\,^2 - f'\,^2)^2 (\alpha^2 f^2 + \beta^2 g^2)^2};
\end{equation}
\begin{equation} \label{E:Eq-H}
H = \frac{\left( f'g'' - f'' g' \right) (\alpha^2 f^2 + \beta^2 g^2) - \left( \alpha^2 f g'+ \beta^2 f'g \right) (g'\,^2 - f'\,^2)}{2(\alpha^2 f^2 + \beta^2 g^2)(g'\,^2-f'\,^2)^{\frac{3}{2}}} \, n_2.
\end{equation}
\vskip 2mm

\vskip 3mm
In a similar way,  we  consider the
surface $\mathcal{M}_2$ in $\mathbb R^4_1$ parametrized by
\begin{equation} \label{E:Eq-8}
\mathcal{M}_2: z(u,v) = \left( f(u) \cos \alpha v, f(u) \sin \alpha v, g(u) \cosh \beta v, g(u) \sinh \beta v \right),
\end{equation}
where $u \in J$, $v \in [0; 2\pi)$, $f(u)$ and $g(u)$ are smooth functions, satisfying the inequalities
$$\alpha^2 f^2(u)- \beta^2 g^2(u) < 0, \quad f'\,^2(u)+ g'\,^2(u) > 0,$$
and $\alpha, \beta$ are positive constants.
The surface, defined by \eqref{E:Eq-8}, is a general rotational surface  for which $x_2(u) =0$ and $x_4(u) =0$. In this case, the meridian curve  $c: x(u) = \left(f(u), 0,  g(u), 0\right)$; $ u \in J \subset \mathbb R$   is spacelike. 

The coefficients of the first fundamental form
of $\mathcal{M}_2$ are
$$E = f'\,^2(u)+ g'\,^2(u); \quad F = 0; \quad G = \alpha^2 f^2(u)- \beta^2 g^2(u).$$
$\mathcal{M}_2$ is a timelike surface in $\mathbb R^4_1$, since $E>0, \; G<0$. We call it  a \textit{general rotational surface of second type}.

We consider the tangent frame field $\{x, y\}$ defined by:
\begin{equation} \label{E:Eq-xy-2}
x = \frac{z_u}{\sqrt{f'\,^2 + g'\,^2}}; \qquad y = \frac{z_v}{\sqrt{\beta^2 g^2 - \alpha^2 f^2}}.
\end{equation}
Obviously, in this case, $\langle x, x \rangle = 1; \, \langle y, y \rangle = - 1; \, \langle x, y \rangle = 0$.
We choose the following normal frame field of $\mathcal{M}_2$:
\begin{equation} \label{E:Eq-normal-2}
\begin{array}{l}
\vspace{2mm}
n_1 = \ds{\frac{1}{\sqrt{f'\,^2 + g'\,^2}}\left(g' \cos \alpha v, g' \sin \alpha v, - f' \cosh \beta v, - f' \sinh \beta v \right)};\\
\vspace{2mm} n_2 = \ds{\frac{1}{\sqrt{\beta^2 g^2 - \alpha^2 f^2}}\left( - \beta g \sin \alpha v, \beta g \cos \alpha v, \alpha
f \sinh \beta v,  \alpha f \cosh \beta v \right)},
\end{array}
\end{equation}
which satisfies $\langle n_1, n_1 \rangle = \langle n_2, n_2 \rangle = 1; \langle n_1, n_2 \rangle =0$.

As in the previous case, by calculating the second partial derivatives of $z(u,v)$ and the functions $c^k_{ij}$, $i,j,k = 1,2$, we obtain the 
second fundamental tensor $\sigma$, which in this case is expressed by:
$$
\begin{array}{l}
\sigma(x, x) = \ds{\frac{f'' g'- f'g''}{(f'\,^2+g'\,^2)^{\frac{3}{2}}} \, n_1};\\
\sigma(x, y) = \ds{\frac{\alpha \beta (f'g- f g')}{\sqrt{f'\,^2+g'\,^2}(\beta^2 g^2 - \alpha^2 f^2)}\, n_2};\\
\sigma(y, y) = - \ds{\frac{\alpha^2 f g'+ \beta^2 f'g}{(f'\,^2 + g'\,^2)\sqrt{\beta^2 g^2 - \alpha^2 f^2}} \, n_1}.
\end{array} 
$$
By use of these formulas we obtain that the Gauss curvature $K$ and the mean curvature vector field $H$ of $\mathcal{M}_2$ are expressed by the functions $f(u), g(u)$ and their derivatives, as follows:
\begin{equation} \label{E:Eq-K2}
K = \ds{\frac{ \alpha^2 \beta^2 \left( f'g - fg'\right)^2 (f'\,^2+ g'\,^2) -  \ds{(f'' g' - f'g'')(\alpha^2 f^2 - \beta^2 g^2)}\left( \alpha^2 f g'+ \beta^2 f'g \right)}{(f'\,^2+ g'\,^2)^2 (\alpha^2 f^2 - \beta^2 g^2)^2}};
\end{equation}
\begin{equation} \label{E:Eq-H2}
H = \ds{\frac{ \left(   f'g'' - f'' g' \right) (\alpha^2 f^2 - \beta^2 g^2) + \left( \alpha^2 f g'+ \beta^2 f'g \right) (f'\,^2+ g'\,^2)  }{2(\alpha^2 f^2 - \beta^2 g^2)(\sqrt{ g'\,^2+f'\,^2})^3}} \,n_1.
\end{equation}

\vskip 2mm
In the following sections we shall describe some basic classes of timelike general rotational surfaces of first and second type, like flat surfaces, surfaces with flat normal connection, minimal surfaces, surfaces of constant mean curvature.

\section{Flat timelike general rotational surfaces}
Let ${\mathcal M}_1$ and ${\mathcal M}_2$ be timelike general rotational surfaces of first and second type, defined by  \eqref{E:Eq-6} and  \eqref{E:Eq-8}, respectively. A surface is called \textit{flat} if the Gauss curvature $K$ is zero. 
In the next statement, we describe analytically all flat timelike general rotational surfaces of first and second type.

\begin{theorem} \label{T:flat}
(i) The timelike general rotational surface of first type  is flat if and only if, up to parametrization, the meridian curve is determined by $c: x(u) = \left(f(u), 0,  0, u\right)$, where $f(u)$ is a solution to the following differential equation:
\begin{equation} \label{E:Eq-GRS1-flat-cond}
\left ( \ln \left \vert \frac{1+f'}{1-f'} \right \vert \right )' =  \frac{-2\alpha^2 \beta^2 \left( f - uf' \right)^2}{\left( \alpha^2 f + \beta^2 uf' \right)(\alpha^2 f^2 + \beta^2 u^2)}.
\end{equation}

(ii) The timelike general rotational surface of second type is flat if and only if, up to parametrization, the meridian curve is determined by $c: x(u) = \left(f(u), 0,  u, 0\right)$, where $f(u)$ is a solution to the following differential equation:
\begin{equation} \label{E:Eq-GRS2-flat-cond}
(\arctan{f'})^{'} = \frac{\alpha^2 \beta^2 \left( uf' - f\right)^2}{(\alpha^2 f^2 - \beta^2 u^2)\left( \alpha^2 f+ \beta^2 uf' \right)}.
\end{equation}
\end{theorem}

\begin{proof}
\textit{(i)} Let ${\mathcal M}_1$  be a  timelike general rotational surface of first  type, defined by  \eqref{E:Eq-6}. 
Using formula  \eqref{E:Eq-K} for the Gauss curvature of ${\mathcal M}_1$, we obtain that $K=0$ if and only if the functions $f(u)$ and $g(u)$ satisfy the equality
\begin{equation} \label{E:Eq-GRS1-flat}
\alpha^2 \beta^2 \left( f g'- f'g \right)^2 (f'\,^2- g'\,^2) = \left( f'' g'- f'g'' \right)\left( \alpha^2 f g'+ \beta^2 f'g \right)(\alpha^2 f^2 + \beta^2 g^2).
\end{equation} 
Without loss of generality we may assume that the meridian curve is parametrized by $f = f(u); \,\, g=u$. 
Then,  equation \eqref{E:Eq-GRS1-flat} takes the form  
$$
\frac{f''}{(1-f'\,^2)}  =  \frac{-\alpha^2 \beta^2 \left( f - uf' \right)^2}{\left( \alpha^2 f + \beta^2 uf' \right)(\alpha^2 f^2 + \beta^2 u^2)},
$$
 which is equivalent to \eqref{E:Eq-GRS1-flat-cond}.

\vskip 2mm
\textit{(ii)} Analogously, if  ${\mathcal M}_2$ is a  timelike general rotational surface of second type, defined by   \eqref{E:Eq-8}, then it follows from \eqref{E:Eq-K2} that $K=0$ if and only if $f(u)$ and $g(u)$ satisfy
\begin{equation} \label{E:Eq-GRS2-flat}
 \alpha^2 \beta^2 \left( f'g - fg'\right)^2 (f'\,^2+ g'\,^2) =  \ds{(f'' g' - f'g'')(\alpha^2 f^2 - \beta^2 g^2)}\left( \alpha^2 f g'+ \beta^2 f'g \right).
\end{equation}
Again we assume that the meridian curve is parametrized by $f = f(u); \,\, g=u$. Then, equation \eqref{E:Eq-GRS2-flat} takes the following form
$$
\frac{f''}{1 + f'\,^2} = \frac{\alpha^2 \beta^2 \left( uf' - f\right)^2}{(\alpha^2 f^2 - \beta^2 u^2)\left( \alpha^2 f+ \beta^2 uf' \right)}, 
$$ 
which is equivalent to  \eqref{E:Eq-GRS2-flat-cond}.
\end{proof}

\section{Timelike general rotational surfaces with flat normal connection}

A surface is said to have \textit{flat normal connection} if the curvature of the normal connection  is zero. 
The curvature of the normal connection $\varkappa$ is expressed by the formula
$\varkappa = \ds{\frac{\langle R^D(x,y, n_1), n_2\rangle}{\langle x, x \rangle  \langle y, y \rangle -\langle x, y \rangle^2}}$,
where $R^D$ is the curvature tensor associated with the normal connection $D$, i.e.
$R^D(x,y,n_1) = D_xD_yn_1 - D_yD_xn_1- D_{[x, y]}n_1$.

\vskip 1mm
Now, let  ${\mathcal M}_1$  be a timelike general rotational surface of first  type, defined by  \eqref{E:Eq-6}. Then, using formula 
\eqref{E:Eq-normal} for the normal frame field of ${\mathcal M}_1$, by direct computations we obtain that:
$$
\begin{array}{l}
\langle \widetilde{\nabla}_x n_1, n_2 \rangle = \langle \widetilde{\nabla}_x n_2, n_1 \rangle = 0; \\
 \langle \widetilde{\nabla}_y n_1, n_2 \rangle = \ds{\frac{\alpha \beta (g g' - f f')}{\sqrt{g'\,^2 -f'\,^2}(\alpha^2 f^2 + \beta^2 g^2)}}; \\ 
\langle \widetilde{\nabla}_y n_2, n_1 \rangle = -\ds{\frac{\alpha \beta (g g' - f f')}{\sqrt{g'\,^2 -f'\,^2}(\alpha^2 f^2 + \beta^2 g^2)}}.
\end{array}
$$
Hence, for the normal connection $D$ of ${\mathcal M}_1$ we have the formulas:
\begin{equation} \label{E:Eq-normal_con}
\begin{array}{ll}
\vspace{2mm}
D_x n_1 = 0; & \quad D_x n_2 = 0;\\
\vspace{2mm}
D_y n_1 = \ds{\frac{\alpha \beta (g g' - f f')}{\sqrt{g'\,^2 -f'\,^2}(\alpha^2 f^2 + \beta^2 g^2)} \, n_2}; & \quad  D_y n_2 = \ds{-\frac{\alpha \beta (g g' - f f')}{\sqrt{g'\,^2 -f'\,^2}(\alpha^2 f^2 + \beta^2 g^2)} \, n_1}.
\end{array} 
\end{equation}
Since the Levi-Civita connection $\widetilde{\nabla}$ is flat, the commutator $[x,y]$ can be calculated by the formula $[x,y] =\widetilde{\nabla}_xy -\widetilde{\nabla}_yx$. 
Having in mind that $x = \ds{\frac{1}{\sqrt{g'\,^2 -f'\,^2}} \,z_u}; \; y = \ds{\frac{1}{\sqrt{\alpha^2 f^2 + \beta^2 g^2}} \,z_v}$, we get 
\begin{equation} \label{E:Eq-com}
[x,y] = \frac{\alpha^2f f' + \beta^2 g g'}{\sqrt{g'\,^2 - f'\,^2}(\alpha^2 f^2 + \beta^2 g^2)}\, y.
\end{equation}
Now, taking into consideration formulas \eqref{E:Eq-normal_con} and \eqref{E:Eq-com}, by long but direct computations we obtain that the curvature of the normal connection $\varkappa$ of ${\mathcal M}_1$ is expressed by the following formula:
\begin{equation} \label{E:Eq-normal_curvature}
\varkappa =\frac{ \alpha \beta ( f g'- f'g) \left((f'\,^2- g'\,^2) (\alpha^2 f g'+ \beta^2 f'g ) + (f'' g'- f'g'')(\alpha^2 f^2 + \beta^2 g^2)\right)}{ (f'\,^2- g'\,^2)^2(\alpha^2 f^2 + \beta^2 g^2)^2}.
\end{equation}

In a similar way, we calculate the curvature of the normal connection of the timelike general rotational surface of second  type 
${\mathcal M}_2$. The normal frame field of ${\mathcal M}_2$ is determined by \eqref{E:Eq-normal-2} and by direct computations we get the following formulas for the normal connection $D$:
\begin{equation} \label{E:Eq-normal_con-2}
\begin{array}{ll}
\vspace{2mm}
D_x n_1 = 0; & \quad D_x n_2 = 0;\\
\vspace{2mm}
D_y n_1 = \ds{\frac{\alpha \beta (f f'+ g g')}{\sqrt{f'\,^2 + g'\,^2}(\beta^2 g^2 - \alpha^2 f^2)} \, n_2}; & \quad  D_y n_2 = \ds{-\frac{\alpha \beta (f f'+ g g')}{\sqrt{f'\,^2 + g'\,^2}(\beta^2 g^2 - \alpha^2 f^2)} \, n_1}.
\end{array} 
\end{equation}
Using \eqref{E:Eq-normal_con-2} we obtain that the  curvature of the normal connection of ${\mathcal M}_2$ is expressed by the formula:
\begin{equation} \label{E:Eq-normal_curvature-2}
\varkappa = \frac{ \alpha \beta (fg' - f'g) \left((f'\,^2+ g'\,^2) (\alpha^2 f g'+ \beta^2 f'g ) + (f'' g'- f'g'')(\alpha^2 f^2 - \beta^2 g^2)\right)}{ (f'\,^2+ g'\,^2)^2(\alpha^2 f^2 - \beta^2 g^2)^2}.
\end{equation}

\vskip 2mm
In the next theorem, we describe analytically all timelike general rotational surfaces of first and second type with flat normal connection.

\begin{theorem} \label{T:flat_normal_connection}
(i) The timelike general rotational surface of first type has flat normal connection if and only if, up to parametrization, the meridian curve is determined by $c: x(u) = \left(f(u), 0,  0, u\right)$, where $f(u)$ is a solution to the following differential equation:
\begin{equation} \label{E:Eq-GRS1-flat-normal-con-cond}
\left ( \ln \left \vert \frac{1+f'}{1-f'} \right \vert \right )' = \frac{2(\alpha^2 f + \beta^2 u f') }{\alpha^2 f^2 + \beta^2 u^2}.
\end{equation}

(ii) The timelike general rotational surface of second type has flat normal connection if and only if, up to parametrization, the meridian curve is determined by $c: x(u) = \left(f(u), 0,  u, 0\right)$, where $f(u)$ is a solution to the following differential equation:
\begin{equation} \label{E:Eq-GRS1-flat-normal-con-cond-2}
(\arctan f')' = \frac{\alpha^2 f + \beta^2 u f'}{\beta^2 u^2  - \alpha^2 f^2  }.
\end{equation} 
\end{theorem}

\begin{proof}
\textit{(i)} Let ${\mathcal M}_1$  be a  timelike general rotational surface of first  type, defined by  \eqref{E:Eq-6}. 
Using formula  \eqref{E:Eq-normal_curvature}, we obtain that the curvature of the normal connection is zero if and only if 
the functions $f(u)$ and $g(u)$ satisfy the equality
\begin{equation} \label{E:Eq-GRS1-flat-normal-con}
\frac{f'' g'- g' f''}{g'\,^2 - f'\,^2 } = \frac{\alpha^2 f g'+ \beta^2 f'g }{\alpha^2 f^2 + \beta^2 g^2}. 
\end{equation}
Without loss of generality we may assume that the meridian curve is parametrized by $f = f(u); \,\, g=u$. 
Then,  equation \eqref{E:Eq-GRS1-flat-normal-con} takes the form  
$$\frac{f''}{1- f'\,^2 } = \frac{\alpha^2 f + \beta^2 u f'}{\alpha^2 f^2 + \beta^2 u^2}, $$
which is equivalent to \eqref{E:Eq-GRS1-flat-normal-con-cond}.

\vskip 2mm
\textit{(ii)} Analogously, if  ${\mathcal M}_2$ is a  timelike general rotational surface of second type, defined by   \eqref{E:Eq-8}, then it follows from \eqref{E:Eq-normal_curvature-2} that $\varkappa = 0$ if and only if $f(u)$ and $g(u)$ satisfy
\begin{equation} \label{E:Eq-GRS2-flat-normal-con}
\frac{f'' g'- f'g''}{f'\,^2+ g'\,^2} = \frac{\alpha^2 f g'+ \beta^2 f'g}{\beta^2 g^2- \alpha^2 f^2 }.
\end{equation}
Again we assume that the meridian curve is parametrized by $f = f(u); \,\, g=u$. Then, equation \eqref{E:Eq-GRS2-flat-normal-con} takes the form
$$\frac{f''}{1 + f'\,^2} = \frac{\alpha^2 f + \beta^2 u f'}{\beta^2 u^2- \alpha^2 f^2 },$$
which is equivalent to \eqref{E:Eq-GRS1-flat-normal-con-cond-2}.
\end{proof}

\section{Minimal timelike general rotational surfaces}

In this section we shall find all minimal timelike general rotational surfaces of first and second type. Recall that a surface is minimal if and only if the normal mean curvature vector field $H$ vanishes. Let  ${\mathcal M}_1$  be a  timelike general rotational surface of first  type, defined by  \eqref{E:Eq-6}.  Using the expression  \eqref{E:Eq-H} for the mean curvature vector field $H$, we 
get that the timelike general rotational surface of  first type is minimal if and only if the functions $f(u)$ and $g(u)$ satisfy the following equality: 
\begin{equation} \label{E:Eq-GRS1-minimal}
\frac{f'' g' - f'g''}{g'^2-f'^2} = \frac{-(\alpha^2fg' + \beta^2gf')}{\alpha^2f^2+\beta^2g^2}.
\end{equation}

Similarly, the general rotational surface of  second type ${\mathcal M}_2$, defined by   \eqref{E:Eq-8}, is minimal if and only if 
$$\frac{f'' g' - f'g''}{f'^2+g'^2} = \frac{(\alpha^2fg' + \beta^2gf')}{\alpha^2f^2-\beta^2g^2}.$$
 
 In the following theorem, we describe explicitly the class of minimal timelike general rotational surfaces of first and second type.

\begin{theorem} \label{T:minimal}
(i) The timelike general rotational surface of first type is minimal if and only if, up to parametrization, the meridian curve is determined by $c: x(u) = \left(f(u), 0,  0, u\right)$, where $f(u)$ is  given by the formula 
\begin{equation} \label{E:Eq-30}
f =  \frac{\sqrt{A}}{\alpha} \sin \left (\varepsilon \frac{\alpha}{\beta} \ln \left \vert \beta u + \sqrt{A+\beta^2 u^2} \right \vert + C \right ); \quad A = const, \, C = const, \, \varepsilon = \pm 1.
\end{equation}
(ii) The timelike general rotational surface of  second type is minimal if and only if, up to parametrization, the meridian curve is determined by $c: x(u) = \left(f(u), 0,  u, 0\right)$, where $f(u)$ is  given by the formula 
\begin{equation} \label{E:Eq-34}
f =  \frac{\sqrt{A}}{\alpha} \sin \left (\varepsilon \frac{\alpha}{\beta} \ln \left \vert \beta u + \sqrt{\beta^2 u^2-A} \right \vert + C \right ); \quad A = const, \, C = const, \, \varepsilon = \pm 1.
\end{equation}
\end{theorem}

\begin{proof}
\textit{(i)} Let  ${\mathcal M}_1$  be a  timelike general rotational surface of first  type, defined by  \eqref{E:Eq-6}. Using  \eqref{E:Eq-xy} and \eqref{E:Eq-second}, we obtain the following  Frenet-type derivative formulas with respect to the  frame field $\{x, y, n_1, n_2\}$, defined by \eqref{E:Eq-xy} and \eqref{E:Eq-normal}: 
$$
\begin{array}{ll}
\vspace{2mm} 
\widetilde{\nabla}_xx = \nu_1\,n_2; & \qquad \quad \widetilde{\nabla}_x n_1 =  -\mu\,y;\\
\vspace{2mm} 
\widetilde{\nabla}_xy =  \mu\,n_1;  & \qquad \quad \widetilde{\nabla}_x n_2 = \nu_1 \,x;\\
\vspace{2mm} 
\widetilde{\nabla}_yx = -\gamma\,y + \mu\,n_1;  & \qquad \quad \widetilde{\nabla}_y n_1= \mu\,x + \varphi\,n_2;\\
\vspace{2mm}
\widetilde{\nabla}_yy = - \gamma\,x + \nu_2\,n_2; & \qquad \quad \widetilde{\nabla}_y n_2= -\nu_2 \,y  -\varphi\,n_1,
\end{array}
$$
where  $\gamma$, $\mu$, $\nu_1$, $\nu_2$, $\varphi$ are smooth functions expressed in terms of $f(u)$ and $g(u)$ as follows:
\begin{equation} \label{Eq-functions-1}
\begin{array}{ll}
\vspace{2mm} \gamma = \ds{-\frac{\alpha^2 f f' + \beta^2 g g'}{\sqrt{g'^2-f'^2}(\alpha^2 f^2 + \beta^2 g^2)}}; &  
\qquad \mu = \ds{\frac{\alpha \beta (fg'-gf')}{\sqrt{g'^2-f'^2}(\alpha^2f^2+\beta^2g^2)}};\\
\vspace{2mm} \nu_1 = \ds{\frac{f'' g' - f'g''}{(\sqrt{g'^2-f'^2})^3}}; & \qquad 
\nu_2 = \ds{-\frac{\alpha^2fg' + \beta^2gf'}{\sqrt{g'^2-f'^2}(\alpha^2f^2+\beta^2g^2)}}; \\
\vspace{2mm} 
\varphi = \ds{\frac{\alpha \beta (gg'-ff')}{\sqrt{g'^2-f'^2}(\alpha^2f^2+\beta^2g^2)}}. &
\end{array}
\end{equation}
Using that the connection $\widetilde{\nabla}$ of $\mathbb R^4_1$ is flat, from $\widetilde{R}(x,y,x) = 0$ and  $\widetilde{R}(x,y,y) = 0$ we obtain that the functions $\gamma$, $\mu$, $\nu_1$, $\nu_2$, $\varphi$ satisfy the following equalities:
\begin{equation} \label{E:Eq-timelike-surf-equalities}
x(\mu) = 2\mu\gamma-\nu_1\varphi; \qquad
x(\nu_2) = \gamma (\nu_1+\nu_2)+\mu\varphi;\qquad
x(\gamma) = \nu_1 \nu_2 -\mu^2 +\gamma^2.
\end{equation}

In the case ${\mathcal M}_1$  is a minimal surface, the functions $\nu_1$ and $\nu_2$ satisfy the relation $\nu_1 = \nu_2$, which is equivalent to \eqref{E:Eq-GRS1-minimal}. In this case, from equalities \eqref{E:Eq-timelike-surf-equalities} we get
$$x(\mu) = 2\mu\gamma-\nu_1\varphi; \qquad
x(\nu_1) = 2\gamma \nu_1+\mu\varphi,$$
which implies that $\gamma = \displaystyle{\frac{1}{4}}\,x\left(\ln(\mu^2 + \nu_1^2)\right)$. 
On the other hand, $\gamma = -\langle \widetilde{\nabla}_y x, y \rangle$ and having in mind \eqref{E:Eq-xy} and $G = \langle z_v,z_v \rangle$, we calculate that $\gamma = -x\left(\ln\sqrt{G}\right)$.
Hence, we obtain the equation 
$$\displaystyle{\frac{1}{4}}x\left(\ln(\mu^2 + \nu_1^2)\right) + x\left(\ln\sqrt{G}\right) = 0,$$
which implies
$$
x\left(G^2(\mu^2 + \nu_1^2)\right) = 0.
$$
Now, using that the functions $\mu$, $\nu_1$, and $G$ depend only on the parameter $u$, we obtain
$G^2(\mu^2 + \nu_1^2) = c^2$
for some real constant $c$. 
Having in mind that  $G = \alpha^2 f^2 + \beta^2 g^2$ and the functions $\mu$ and $\nu_1$ are expressed in terms of $f$ and $g$ as given in \eqref{Eq-functions-1},  we obtain the equality 
$$
\ds{\frac{\alpha^2 \beta^2 (fg'-gf')^2 + (\alpha^2fg' + \beta^2gf')^2}{g'^2-f'^2}} = c^2,
$$
which implies
$$\ds{\frac{\alpha^2 f^2 g'^2 + \beta^2 g^2 f'^2}{g'^2-f'^2}} = \frac{c^2}{\alpha^2 + \beta^2}.$$
Without loss of generality we may assume that $g'^2-f'^2 = 1$, i.e. the meridian curve $c$ is parametrized by the arc-length. So, we get
\begin{equation} \label{E:Eq-28}
\alpha^2 f^2 g'^2 + \beta^2 g^2 f'^2 = \frac{c^2}{\alpha^2 + \beta^2}.
\end{equation}
Denote $A = \ds{\frac{c^2}{\alpha^2 + \beta^2}}$. Now, using that $g'^2 = 1+f'^2$, from \eqref{E:Eq-28} we get
\begin{equation} \label{E:Eq-29}
f'^2 = \frac{A-\alpha^2 f^2}{\alpha^2 f^2+\beta^2 g^2}; \qquad g'^2 = \frac{A+\beta^2 g^2}{\alpha^2 f^2+\beta^2 g^2}.
\end{equation}
Note that the constant $A$ satisfies $A > \alpha^2 f^2$. Equalities \eqref{E:Eq-29} imply that 
$$(A+\beta^2 g^2)f'^2 = (A-\alpha^2 f^2)g'^2,$$
which is equivalent to 
$$\frac{f'}{\sqrt{A-\alpha^2 f^2}} = \varepsilon \frac{g'}{\sqrt{A+\beta^2 g^2}}, \quad \varepsilon = \pm 1.$$
Thus, integrating the last equality, we obtain:
$$
\int{\frac{df}{\sqrt{A-\alpha^2 f^2}}} = \varepsilon \int{\frac{dg}{\sqrt{A+\beta^2 g^2}}}
$$
and calculating these  integrals, we get
$$
\arcsin \ds{\frac{\alpha f}{\sqrt{A}}} = \varepsilon \frac{\alpha}{\beta} \ln \left \vert \beta g + \sqrt{A+\beta^2 g^2} \right \vert + C, \quad C = const.
$$
Now, setting $g =u$, we obtain that the function $f$ satisfies:
$$
\arcsin \ds{\frac{\alpha f}{\sqrt{A}}} = \varepsilon \frac{\alpha}{\beta} \ln \left \vert \beta u + \sqrt{A+\beta^2 u^2} \right \vert + C, \quad C = const.
$$

Consequently, in the case of a minimal  timelike general rotational surface of  first type, the meridian curve $c$ is given by formula \eqref{E:Eq-30}.

\vskip 2mm 
\textit{(ii)} In a similar way, for the timelike general rotational surface of second type
we obtain the following  Frenet-type derivative formulas with respect to the  frame field $\{x, y, n_1, n_2\}$, defined by \eqref{E:Eq-xy-2} and  \eqref{E:Eq-normal-2}: 
$$
\begin{array}{ll}
\vspace{2mm} 
\widetilde{\nabla}_xx =  \nu_1\,n_1; & \qquad \quad \widetilde{\nabla}_x n_1 =  -\nu_1\,x;\\
\vspace{2mm} 
\widetilde{\nabla}_xy =   \mu\,n_2;  & \qquad \quad \widetilde{\nabla}_x n_2 = \mu \,y;\\
\vspace{2mm} 
\widetilde{\nabla}_yx = \gamma\,y + \mu\,n_2;  & \qquad \quad \widetilde{\nabla}_y n_1= \nu_2\,y + \varphi\,n_2;\\
\vspace{2mm}
\widetilde{\nabla}_yy = \gamma\,x + \nu_2\,n_1; & \qquad \quad \widetilde{\nabla}_y n_2= -\mu \,x  -\varphi\,n_1,
\end{array}
$$
where  $\gamma$, $\mu$, $\nu_1$, $\nu_2$, $\varphi$ are smooth functions expressed in terms of $f(u)$ and $g(u)$ as follows:
\begin{equation} \label{Eq-functions-2}
\begin{array}{ll}
\vspace{2mm} \gamma = \ds{\frac{\alpha^2 f f' - \beta^2 g g'}{\sqrt{f'^2+g'^2}(\alpha^2 f^2 - \beta^2 g^2)}}; &  
\qquad \mu = \ds{\frac{\alpha \beta (fg'-gf')}{\sqrt{f'^2+g'^2}(\alpha^2f^2-\beta^2g^2)}};\\
\vspace{2mm} \nu_1 = \ds{\frac{f'' g' - f'g''}{(\sqrt{f'^2+g'^2})^3}}; & \qquad 
\nu_2 = \ds{\frac{\alpha^2fg' + \beta^2gf'}{\sqrt{f'^2+g'^2}(\alpha^2f^2-\beta^2g^2)}}; \\
\vspace{2mm} 
\varphi = \ds{-\frac{\alpha \beta (gg'+ff')}{\sqrt{f'^2+g'^2}(\alpha^2f^2-\beta^2g^2)}}. &
\end{array}
\end{equation}
Again, using that the connection $\widetilde{\nabla}$ of $\mathbb R^4_1$ is flat, from $\widetilde{R}(x,y,x) = 0$ and  $\widetilde{R}(x,y,y) = 0$ we obtain  the following equalities in the case of a minimal surface:
$$
x(\mu) = \nu_1\varphi - 2\mu\gamma; \qquad
x(\nu_1) = -2\gamma\nu_1-\mu\varphi,
$$
which imply $\gamma = \displaystyle{-\frac{1}{4}\, x\left(\ln(\mu^2 + \nu_1^2)\right)}$. On the other hand, it can easily be calculated that 
$\gamma = x\left(\ln\sqrt{-G}\right)$. Hence, we get $\displaystyle{\frac{1}{4}} \,x\left(\ln(\mu^2 + \nu_1^2)\right) + x\left(\ln\sqrt{-G}\right) = 0$, which implies
$$
x\left(G^2(\mu^2 + \nu_1^2)\right) = 0.
$$
Since $\mu$, $\nu_1$, and $G$ are functions depending only on $u$, we obtain
$G^2(\mu^2 + \nu_1^2) = c^2,$
where $c$ is a constant. In the case of a timelike general rotational surface of second type we have  $G = \alpha^2 f^2 - \beta^2 g^2$, and using the expressions of $\mu$ and $\nu_1$ given in  \eqref{Eq-functions-2}, we obtain:
$$
\ds{\frac{\alpha^2 \beta^2 (f'g - fg')^2 + (\alpha^2fg' + \beta^2gf')^2}{(f'^2 + g'^2)}} = c^2,
$$
which is equivalent to
$$
\ds{\frac{\alpha^2 f^2 g'^2 + \beta^2 g^2 f'^2}{(f'^2 + g'^2)}} = \frac{c^2}{\alpha^2 + \beta^2}.
$$
Without loss of generality we assume that $f'^2 + g'^2 = 1$, i.e. the meridian curve is parametrized by the arc-length. So, we get
$$\alpha^2 f^2 g'^2 + \beta^2 g^2 f'^2 = \frac{c^2}{\alpha^2 + \beta^2}.$$
Denote $A = \ds{\frac{c^2}{\alpha^2 + \beta^2}}$. Using that $g'^2 = 1-f'^2$, we obtain  
\begin{equation} \label{E:Eq-33}
f'^2 = \frac{A-\alpha^2 f^2}{\beta^2 g^2-\alpha^2 f^2}, \quad g'^2 = \frac{\beta^2 g^2-A}{\beta^2 g^2-\alpha^2 f^2}.
\end{equation} 
The constant $A$ satisfies $\alpha^2 f^2 < A < \beta^2 g^2$. It follows from \eqref{E:Eq-33} that 
$$(\beta^2 g^2-A)f'^2 = (A-\alpha^2 f^2)g'^2,$$
which is equivalent to
$$\frac{f'}{\sqrt{A-\alpha^2 f^2}} = \varepsilon \frac{g'}{\sqrt{\beta^2 g^2-A}}, \quad \varepsilon = \pm 1.$$
Integrating the last equality, we obtain:
$$
\int{\frac{d f}{\sqrt{A-\alpha^2 f^2}}} = \varepsilon \int{\frac{dg}{\sqrt{\beta^2 g^2-A}}},
$$
and calculating the integrals, we get
$$
\arcsin \ds{\frac{\alpha f}{\sqrt{A}}} = \varepsilon \frac{\alpha}{\beta} \ln \left \vert \beta g + \sqrt{\beta^2 g^2-A} \right \vert + C, \quad C = const.
$$
Consequently, after setting $g =u$, we obtain that in the case of a minimal timelike  general rotational surface of second type, the meridian curve is given by formula 
\eqref{E:Eq-34}. 
\end{proof}

\section{Timelike general rotational surfaces with constant mean curvature}

In this section we shall classify the timelike general rotational surfaces with non-zero constant mean curvature, i.e. $\langle H, H \rangle = const \neq 0$.
 The mean curvature vector field $H$ of a timelike general rotational surface of first type is expressed by formula \eqref{E:Eq-H}.
Hence, 
$$\langle H, H \rangle = \ds{\frac{ \left ( \left( f'g'' - f'' g' \right) (\alpha^2 f^2 + \beta^2 g^2) + \left( \alpha^2 f g'+ \beta^2 f'g \right) (f'\,^2- g'\,^2) \right )^2 }{4(\alpha^2 f^2 + \beta^2 g^2)^2( g'\,^2-f'\,^2)^3}}.$$

So, $\langle H, H \rangle = const$ if and only if the following equality is satisfied: 
\begin{equation} \label{E:Eq-39}
\ds{\frac{f'' g' - f'g''}{(\sqrt{g'^2-f'^2})^3}} = -\ds{\frac{\alpha^2fg' + \beta^2gf'}{(\alpha^2f^2+\beta^2g^2)(\sqrt{g'^2-f'^2})}} + c, \quad c  \neq 0.
\end{equation}

Assume that the meridian curve is parametrized by $f = f(u); g = u$. Then, equality  \eqref{E:Eq-39} takes the form
$$\ds{\frac{f''}{(\sqrt{1-f'^2})^3}} = - \ds{\frac{\alpha^2f + \beta^2uf'}{(\alpha^2f^2+\beta^2u^2)(\sqrt{1-f'^2})}} + c,$$
which is equivalent to
$$
\ds{\frac{f''}{1-f'^2}} = -\ds{\frac{\alpha^2f + \beta^2uf'}{\alpha^2f^2+\beta^2u^2}} + c \sqrt{1-f'^2},
$$
or equivalently
\begin{equation} \label{E:Eq-41}
\left ( \ln \left \vert \frac{1+f'}{1-f'}\right \vert \right )' = - 2 \ds{\frac{\alpha^2f + \beta^2uf'}{\alpha^2f^2+\beta^2u^2}} + 2 c \sqrt{1-f'^2}.
\end{equation} 

Similarly, the mean curvature vector field of a timelike general rotational surface of second type is expressed by  \eqref{E:Eq-H2}. Hence, $\langle H, H \rangle = const$ if and only if
$$\ds{\frac{f'' g' - f'g''}{(\sqrt{f'^2+g'^2})^3}} = \ds{\frac{(\alpha^2fg' + \beta^2gf')}{(\alpha^2f^2-\beta^2g^2)(\sqrt{f'^2 + g'^2})}} + c, \quad c \neq 0.$$
Assume that the meridian curve is parametrized by $f = f(u); g = u$. Then, the above  equation  takes the form 
$$
\ds{\frac{f''}{1+f'^2}} = \ds{\frac{\alpha^2f + \beta^2uf'}{\alpha^2f^2-\beta^2u^2}} + c \sqrt{1+f'^2},
$$
which is equivalent to
\begin{equation} \label{E:Eq-43}
({\rm arctan} f')' = \ds{\frac{\alpha^2f + \beta^2uf'}{\alpha^2f^2-\beta^2u^2}} + c \sqrt{1+f'^2}.
\end{equation}

Equations \eqref{E:Eq-41} and \eqref{E:Eq-43} describe analytically the class of timelike general rotational surfaces of  first and second type with constant mean curvature.

Finally, we proved the following result.

\begin{theorem} \label{T:CMC}
(i) The timelike general rotational surface of first type  has non-zero constant mean curvature if and only if, up to parametrization, the meridian curve is determined by $c: x(u) = \left(f(u), 0,  0, u\right)$, where $f(u)$ is a solution to the following differential equation:
\begin{equation*} 
\left ( \ln \left \vert \frac{1+f'}{1-f'}\right \vert \right )' = - 2 \ds{\frac{\alpha^2f + \beta^2uf'}{\alpha^2f^2+\beta^2u^2}} + 2 c \sqrt{1-f'^2}, \quad c \neq 0.
\end{equation*} 

(ii) The timelike general rotational surface of second type has non-zero constant mean curvature if and only if, up to parametrization, the meridian curve is determined by $c: x(u) = \left(f(u), 0,  u, 0\right)$, where $f(u)$ is a solution to the following differential equation:
\begin{equation*} 
({\rm arctan} f')' = \ds{\frac{\alpha^2f + \beta^2uf'}{\alpha^2f^2-\beta^2u^2}} + c \sqrt{1+f'^2}, \quad c \neq 0.
\end{equation*}
\end{theorem}

\section{Timelike general rotational surfaces with parallel normalized mean curvature vector field} \label{S:parallel}

In this section, we shall give explicitly all timelike general rotational surfaces with parallel normalized mean curvature vector field. Let us recall that a surface  has \textit{parallel normalized mean curvature vector field} if the mean curvature vector field $H$ is non-zero and  there exists a unit vector field in the direction of $H$ 
which is parallel.

\vskip 2mm
Now, let $\mathcal{M}_1$ be  a  timelike general rotational surface of first  type, defined by  \eqref{E:Eq-6}. In the case $H \neq 0$ the normalized mean curvature vector field of $\mathcal{M}_1$ is $n_2$, where $n_2$ is determined in \eqref{E:Eq-normal}. 
So, $\mathcal{M}_1$ has parallel normalized mean curvature vector field if and only if the following equalities are satisfied: $D_x n_2 = 0$ and $D_y n_2 = 0$.

Similarly, if $\mathcal{M}_2$ is a timelike general rotational surface of second type, defined by \eqref{E:Eq-8}, then, $\mathcal{M}_2$ has parallel normalized mean curvature vector field if and only if $D_x n_1 = 0$ and $D_y n_1 = 0$, where $n_1$ is given in \eqref{E:Eq-normal-2}. 

The next theorem describes all timelike general rotational surfaces of first and second type with parallel normalized mean curvature vector field.

\begin{theorem} \label{T:parallel}
(i) The timelike general rotational surface of first type  has parallel normalized mean curvature vector field if and only if, up to parametrization, the meridian curve is determined by $c: x(u) = \left(f(u), 0,  0, u\right)$, where 
\begin{equation*}
f(u) = \pm \sqrt{u^2 + C^2}; \quad C = const \neq 0.
\end{equation*}

(ii) The timelike general rotational surface of second type has parallel normalized mean curvature vector field  if and only if, up to parametrization, the meridian curve is determined by $c: x(u) = \left(f(u), 0,  u, 0\right)$, where 
\begin{equation*}
f(u) = \pm \sqrt{C^2 - u^2}; \quad u \in (-C, C); \quad C = const \neq 0.
\end{equation*}
\end{theorem}

\begin{proof}
\textit{(i)} Let  ${\mathcal M}_1$  be a  timelike general rotational surface of first  type, defined by  \eqref{E:Eq-6}. In Section 5, we derived formulas  \eqref{E:Eq-normal_con} giving the  normal connection $D$ of ${\mathcal M}_1$. These formulas  imply that 
 $D_x n_2 = D_y n_2 = 0$ if and only if the functions $f(u)$ and $g(u)$ satisfy the following differential equation:
$$f f'- g g'= 0.$$ 
The above equation implies that the functions $f$ and $g$ are related by  $f^2 = g^2 + C_1$ for some constant $C_1$. Without loss of generality we may assume that $g(u) = u$. Then  $f(u) = \pm \sqrt{u^2 + C_1}$. Since $f'^2 - g'^2 <0$, we obtain $C_1 > 0$. 
Hence, $f(u) = \pm \sqrt{u^2 + C^2}$ for some constant $C \neq 0$.

\vskip 2mm
\textit{(ii)} In a similar way, having in mind formulas \eqref{E:Eq-normal_con-2} giving the  normal connection $D$ of ${\mathcal M}_2$, we obtain that the timelike general rotational surface of second type has parallel normalized mean curvature vector field if and only if the  functions $f(u)$ and $g(u)$ satisfy the following differential equation:
$$f f'+ g g'= 0,$$ 
which implies that  $f^2 + g^2 = C_1$ for some constant $C_1$. Again we  assume that $g(u) = u$. 
Then  $f(u) = \pm \sqrt{- u^2 + C_1}$. Since $f'^2 + g'^2 >0$, we obtain $C_1 > 0$. Hence, $f(u) = \pm \sqrt{C^2 - u^2}$ for some constant $C \neq 0$.
\end{proof}

\vskip 6mm 
\textbf{Acknowledgments:}
The  authors are partially supported by the National Science Fund,
Ministry of Education and Science of Bulgaria under contract KP-06-N52/3.

\end{document}